\documentclass[12pt,twoside]{amsart}

\usepackage{amsmath}
\usepackage{amssymb}
\usepackage{amscd}
\usepackage{graphics}
\usepackage{graphicx}

\renewcommand{\bar}{\overline}

\newcommand{\eps}{\varepsilon}

\newcommand{\CC}{\mathbb{C}}

\newcommand{\FF}{\mathbb{F}}

\newcommand{\PP}{\mathbb{P}}

\newcommand{\RR}{\mathbb{R}}

\newcommand{\ZZ}{\mathbb{Z}}

\newcommand{\Cv}{\CC_v}

\newcommand{\ints}{{\mathcal O}}

\newcommand{\maxid}{{\mathcal M}}

\newcommand{\calF}{{\mathcal F}}

\newcommand{\calJ}{{\mathcal J}}

\newcommand{\Kbar}{\overline{K}}

\newcommand{\PCv}{\PP^1(\Cv)}

\newcommand{\PK}{\PP^1(K)}
\newcommand{\PL}{\PP^1(L)}
\newcommand{\PKbar}{\PP^1(\overline{K})}

\newcommand{\Ber}{\textup{Ber}}
\newcommand{\PBerk}{\PP^1_{\Ber}}

\DeclareMathOperator{\PGL}{PGL}

\DeclareMathOperator{\ord}{ord}
\DeclareMathOperator{\lcm}{lcm}

\newcommand{\Dbar}{\bar{D}}

\newcommand{\dsps}{\displaystyle}

\theoremstyle{plain}
\newtheorem{thm}{Theorem}[section]

\newtheorem{lemma}[thm]{Lemma}
\newtheorem{conj}{Conjecture}

\newtheorem*{thmA}{Theorem A}
\newtheorem*{thmB}{Theorem B}

\theoremstyle{definition}
\newtheorem{defin}[thm]{Definition}

\theoremstyle{remark}
\newtheorem{remark}[thm]{Remark}

\numberwithin{equation}{section}

\title[Attaining Potentially Good Reduction]
{Attaining Potentially Good Reduction in Arithmetic Dynamics}
\date{October 25, 2014}
\subjclass[2010]{Primary: 37P05 Secondary: 37P20, 11S82}
\keywords{arithmetic dynamics, good reduction, periodic points}

\author{Robert~L. Benedetto}
\address{Amherst College \\ Amherst, MA 01002, USA}
\email{rlbenedetto@amherst.edu}

\begin{document}

\newcounter{bean}
\newcounter{sheep}

\begin{abstract}
Let $K$ be a non-archimedean field, and
let $\phi\in K(z)$ be a rational function of degree $d\geq 2$.
If $\phi$ has potentially good reduction, we give an upper
bound, depending only on $d$, for the minimal degree of
an extension $L/K$ such that $\phi$ is conjugate over $L$ to
a map of good reduction.  In particular, if $d=2$ or $d$ is
less than the residue characteristic of $K$, the bound is $d+1$.
If $K$ is discretely valued, we give examples to
show that our bound is sharp.
\end{abstract}

\maketitle

Fix the following notation throughout this paper.
\begin{tabbing}
\hspace{1cm} \= \hspace{2cm} \=  \kill
\> $K$: \> a field \\
\> $\Kbar$: \> an algebraic closure of $K$ \\
\> $|\cdot|$: \> a non-archimedean absolute value on $\Kbar$ \\
\> $\Cv$: \> the completion of $\Kbar$ with respect to $|\cdot|$ \\
\> $\ints_K$: \> the ring of integers $\{x\in K : |x| \leq 1\}$ of $K$ \\
\> $\maxid_K$: \> the maximal ideal $\{x\in K : |x| < 1\}$ of $\ints_K$ \\
\> $k$: \> the residue field $\ints_K/ \maxid_K$ of $K$
\end{tabbing}

%\> $K_v$: \> the completion of $K$ with respect to $|\cdot|_v$ \\
%\> $L$: \> an algebraic extension of $K$ \\
%\> $|\cdot|_w$: \> an extension of $|\cdot|_v$ to $L$ \\
%\> $L_w$: \> the completion of $L$ with respect to $|\cdot|_w$ \\
%For example, $K$ could be the field $\Qp$ of $p$-adic rationals,
%with ring of integers $\Zp$, maximal ideal $p\Zp$, and algebraic
%closure $\Qpbar$.  Since $\Qp$ is complete, the absolute value
%$|\cdot|_p$ on $\Qp$ extends uniquely to $\Qpbar$.
%\> $\ints_L$: \> the ring of integers $\{x\in L : |x|_w \leq 1\}$ of $L$ \\
%\> $\maxid_L$: \> the maximal idea $\{x\in L : |x|_w < 1\}$ of $\ints_L$ \\
%\> $\ell$: \> the residue field $\ints_L/ \maxid_L$ of $L$

Let $\phi(z)\in K(z)$ be a rational function.
We define the degree of $\phi=f/g$ to be
$\deg\phi:=\max\{\deg f, \deg g\}$,
where $f,g\in K[z]$ have no common factors.
We will view $\phi$ as a a dynamical system acting on
$\PCv=\Cv\cup\{\infty\}$.
For a thorough treatment of dynamics over such non-archimedean fields,
see Chapter~10 of \cite{BR}, or \cite{BenAZ}.
%Thus, we define the $n$-iterate $\phi^n$ of $\phi$
%under composition by defining $\phi^0(z):=z$, and for
%every $n\geq 1$, setting $\phi^n(z):=\phi\circ\phi^{n-1}(z)$.

The notion of \emph{good reduction} of $\phi$ first
appeared in \cite{MS1}; see Definition~\ref{def:goodred}.
We say that $\phi$ has \emph{potentially good reduction}
if $\phi$ is conjugate over $L$ to a map of good reduction,
for some finite extension $L/K$.
In \cite{BenPG1}, we gave a necessary and sufficient condition
for whether or not $\phi$ has potentially good reduction.
In this paper, we turn to a related question: if $\phi$ has
potentially good reduction, how large an extension $L/K$
is required to attain good reduction?
In Theorem~3.6 of \cite{Rum},
Rumely showed that when $\deg\phi=d\geq 2$
and $\phi$ has potentially good reduction,
there is an extension $L/K$ of degree at most $(d+1)^2$
such that $\phi$ is conjugate over $L$ to a map of good reduction.
In this paper, we improve Rumely's bound, as follows.

\begin{thmA}
Let $K$ be a field with non-archimedean absolute value $|\cdot|$
and residue characteristic $p\geq 0$,
and let $\phi\in K(z)$ be a rational function of degree $d\geq 2$.
Define
$$B = B(p,d) := \begin{cases}
p^e (d-1) & \text{ if } d\geq 3 \text{ and } d=mp^e
\\
& \quad \text{ for integers } e,m\geq 1 \text{ with } p\nmid m,
\\
dp^e & \text{ if } d\geq 3 \text{ and } d=1 + mp^e
\\
& \quad \text{ for integers } e,m\geq 1 \text{ with } p\nmid m,
\\
d+1 & \text{ otherwise}.
\end{cases}
$$
If $\phi$ has potentially good reduction, then there is an extension
$L/K$ with $[L:K] \leq B$
such that $\phi$ is conjugate over $L$ to a map of good reduction.
\end{thmA}

For $d\geq 3$, the bound $B$ of the Theorem~A is at most $d(d-1)$.
Moreover, if $d=2$, or if $p>d$, then $B$ is simply $d+1$.
In addition, $B$ is a sharp bound, at least for many choices
of the field $K$.

\begin{thmB}
Let $K$ be a field with
non-archimedean absolute value $|\cdot|$
and residue characteristic $p\geq 0$,
and let $d\geq 2$ be an integer.
Define the integer $B$ as in Theorem~A.
Suppose $K$ contains an element $\pi$ such that
for any integer $n\geq 2$ with $n|B$,
$|\pi|^{1/n}\not\in |K|$.
Then there is a rational function $\phi\in K(x)$ of degree $d$
and potentially good reduction
such that for any extension field $L$ with $[L:K]<B$,
$\phi$ is not conjugate over $L$ to a map of good reduction.
\end{thmB}

The existence of the element $\pi$ in the
hypothesis of Theorem~B certainly holds
if $K$ is discretely valued.  After all, in
that case, we may choose $\pi$ to be a uniformizer for $K$.

%\begin{thmB}
%Let $K$ be a field with discretely valued
%non-archimedean absolute value $|\cdot|$
%and residue characteristic $p\geq 0$,
%and let $d\geq 2$ be an integer.
%Define $B$ as in Theorem~A.
%Then there is a rational function $\phi\in K(x)$ of degree $d$
%and potentially good reduction
%such that for any extension field $L$ with $[L:K]<B$,
%$\phi$ is not conjugate over $L$ to a map of good reduction.
%\end{thmB}

The outline of the paper is as follows.
We will recall some general facts about arithmetic dynamics
and the Berkovich projective line in Section~\ref{sec:background}.
In Section~\ref{sec:lemmas}, we will state and prove some
auxilliary results that we will need.
We will then prove Theorem~A in Section~\ref{sec:proofA},
and Theorem~B in Section~\ref{sec:proofB}.
Finally, in Section~\ref{sec:ramify}, we will discuss
a conjectural strengthening of Theorem~A, that the extension $L/K$
can always be chosen to be totally ramified.

\section{Background}
\label{sec:background}

If $n\in\ZZ$ is a nonzero integer and $p\geq 2$ is prime,
we use the standard notation $v_p(n)$ to indicate the
largest integer $e\geq 0$ such that $p^e|n$.  For $p=0$,
we will use the following conventions:
$v_p(n)=0$ for all $n\in\ZZ\smallsetminus\{0\}$, and $p^0=1$.

A map $\phi\in K(z)$ of degree $d=\deg\phi\geq 2$
has exactly $d+1$ fixed points
in $\PCv$, counted with appropriate multiplicity.
If $x\in\Cv$ is such a fixed point, its
\emph{multiplier} is $\lambda:=\phi'(x)$.
Given $h\in\PGL(2,\Cv)$, $h(x)$ is a fixed point
of $\psi:=h\circ\phi\circ h^{-1}$, and the multiplier
of $h(x)$ under $\psi$ is also $\lambda$.  Thus,
we can define the multiplier of a fixed point at $\infty$
by changing coordinates via such a conjugation.

Note that $\lambda=1$ if and only if $x$ has multiplicity
at least two as a fixed point of $\phi$,
and that $\lambda=0$
if and only if $x$ is a critical point of $\phi$.
We say that $x$ is \emph{attracting} is $|\lambda|<1$,
\emph{repelling} if $|\lambda|>1$,
and \emph{indifferent} if $|\lambda|=1$.

Given a polynomial $f(z)\in \ints_{K}[z]$, denote by
$\overline{f}(z)\in k[z]$ the polynomial formed by
reducing all coefficients of $f$ modulo $\maxid_{K}$.

\begin{defin}
\label{def:goodred}
Let $\phi(z)\in K(z)$ be a rational function
of degree $d\geq 1$.  Write $\phi=f/g$ with
$f,g\in\ints_K[z]$ and with at least one coefficient
of $f$ or $g$ having absolute value $1$.
Let $\overline{\phi}:=\overline{f}/\overline{g}\in k(z)\cup\{\infty\}$.
We say that $\phi$ has {\em good reduction}
if $\deg \overline{\phi} = \deg \phi$.
Otherwise,
we say that $\phi$ has \emph{bad reduction}.

We say that $\phi$ has \emph{potentially good reduction}
if there is a finite extension $L/K$ in $\Kbar$
and a map $h\in\PGL(2,L)$ such that
$h\circ\phi\circ h^{-1}\in L(z)$ has good reduction.
\end{defin}

%For example, if $a\in K$ is a uniformizer, then the map
%$\phi(z) = az^2$ has bad reduction, but its
%$K$-rational conjugate $a\phi(a^{-1}z)=z^2$ has good reduction.
%Meanwhile, the map $\psi(z)=az^3$ also has bad reduction,
%and in fact it can be shown to have bad reduction even after
%any $K$-rational coordinate change.  However, $\psi$ has
%potentially good reduction, because
%$a^{1/2}\psi(a^{-1/2}z) = z^3$ has good reduction.

It is easy to see that for any $\phi\in K(z)$,
polynomials $f,g\in\ints_K[z]$ exist with $\phi=f/g$
and with at least one coefficient of $f$ or $g$ having
absolute value $1$.  Moreover, the reduction type
(good or bad) of $\phi$ is independent of the choice
of the pair $f,g$.
%In addition, $\phi$ has bad reduction if and only if
%$\deg \overline{\phi} < \deg \phi$.
If $\phi\in K[z]$ is a polynomial, then it has good reduction
if and only if $\phi\in\ints_K[z]$ and the lead coefficient
of $\phi$ has absolute value $1$.

%The reduction map $\red:\ints_{\overline{K}}\to \overline{k}$
%induces a map $\red:\PKbar\to\Pkbar$, which coincides with the
%original reduction map on $\ints_{\overline{K}}$ and maps
%$\PKbar\smallsetminus\ints_{\overline{K}}$ to the point
%$\overline{\infty}\in\Pkbar$.  It is easy to check that
%a rational function $\phi\in\Kbar(z)$ has good reduction
%if and only if
%% it respects the reduction map, i.e., if
%$$\overline{\phi}\big(\bar{x}\big) = \overline{\phi(x)}
%\quad\text{for all}\quad x\in\PKbar.$$
%As a result, the composition of two maps of good reduction
%again has good reduction.
%In addition, among maps of degree~1, i.e., in $\PGL(2,\Kbar)$,
%the good reduction maps are precisely the elements of
%$\PGL(2,\ints_{\Kbar})$.

Given $a\in\Cv$ and $r>0$,
$$D(a,r) := \{x\in\Cv : |x-a|<r\}$$
and
$$\Dbar(a,r) := \{x\in\Cv : |x-a|\leq r\}$$
will denote the open and closed disks, respectively,
in $\Cv$ containing $a$ and of radius $r$.
If $r\in |\Cv^{\times}|$, we say that the above
disks are \emph{rational}.

If $\phi\in\Cv(z)$ is a nonconstant rational function,
and if $D(a,r)\subseteq\Cv$ is a rational open disk
containing no poles of $\phi$, then we may write
$$\phi(z) = \sum_{n\geq 0} c_n (z-a)^n$$
as a power series converging on $D(a,r)$.
In that case, $\{|c_n|r^n: n\geq 0\}$ is bounded
and attains its maximum.  Moreover, this power series
has an associated \emph{Weierstrass degree}, which
is the smallest integer $\ell\geq 0$ such that
$$|c_{\ell}|r^{\ell} = \max\{ |c_n| r^n : n\geq 0\}.$$
It is a consequence of the Weierstrass Preparation Theorem
(and more specifically, of the theory of Newton polygons)
that in that case, the image $\phi(D(a,r))$ is also
a rational open disk.
Moreover, if that image disk $D(b,s)$ contains the point $0$,
then the Weierstrass degree $\ell$ of $\phi$ on $D(a,r)$ is 
at least $1$, and the mapping
$$\phi: D(a,r) \to D(b,s)
\quad \text{is everywhere $\ell$-to-1},$$
meaning that every point $y\in D(b,s)$ has exactly $\ell$
preimages in $D(a,r)$, counting multiplicity.
In particular, we must have $1\leq \ell\leq \deg\phi$.

%GIVE REFERENCES FOR THE ABOVE!!!!

%(FROM POLY GOOD RED THROUGH DISKS STUFF.)

The Berkovich projective line $\PBerk$ over $K$ is a certain
compact Hausdorff topological space
containing $\PCv$ as a subspace.  
The precise definition,
which uses multiplicative seminorms on $\Cv$-algebras,
is rather involved; the interested reader may consult
Berkovich's original monograph \cite{Ber}, the detailed
exposition in \cite{BR}, or the summary in \cite{BenAZ},
for example.
%\footnote{Readers familiar with $\PBerk$
%may note that the construction is usually done over the completion
%$\Cv$ of $\Kbar$.  Because
%our absolute value $|\cdot|$ is already defined
%on all of $\Kbar$, however, the multiplicative seminorms
%over $\Kbar$ coincide with those over $\Cv$.}
We only state some basic properties here,
without proofs.

The space $\PBerk$ is uniquely path-connected:
given any two distinct points
$\xi_0,\xi_1\in\PBerk$, there is a \emph{unique} arc
between them.  That is, there there is a unique subspace
$X\subseteq\PBerk$ homeomorphic to the interval
$[0,1]\subseteq\RR$,
where $\xi_0,\xi_1\in X$, and the homeomorphism takes
$\xi_0$ to $0$ and $\xi_1$ to $1$.

For each closed disk $\Dbar(a,r)\subseteq\Cv$, the space
$\PBerk$ contains a unique associated point,
which we shall denote $\zeta(a,r)$.
If $r\in |\Cv^{\times}|$, then $\zeta(a,r)$ is said to be
a \emph{type~II point}.  For completeness, we note that the
points $\zeta(a,r)$ with $r>0$ but $r\not\in |\Cv^{\times}|$
are of type~III.  The type~I points are simply points of $\PP^1(\Cv)$.
There are also type~IV points, corresponding to decreasing
chains of disks in $\Cv$ with empty intersection.  However,
in this paper we will only be concerned with the type~II points.

For any type~II point $\zeta(a,r)\in\PBerk$, the complement
$\PBerk\smallsetminus\{\zeta(a,r)\}$ is no longer connected, but
instead consists of infinitely many connected components.
When intersected with $\PCv$, one of these components
is $\PCv\smallsetminus\Dbar(a,r)$, while the rest are
the infinitely many rational open disks
$D(b,r)\subseteq\PCv$ with $b\in\Dbar(a,r)$.

Finally, any rational function $\phi\in\Cv(z)$
extends uniquely to a continuous function $\phi:\PBerk\to\PBerk$.
In particular, if $D(a,r)\subseteq\Cv$ is a rational open disk
containing no poles of $\phi$, then
$\phi(\zeta(a,r))=\zeta(b,s)$, where $D(b,s)=\phi(D(a,r))$
is the image in $\Cv$ of $D(a,r)$ under $\phi$.

\section{Some Lemmas}
\label{sec:lemmas}

We say a point $\xi\in\PBerk$ is
\emph{totally invariant} under $\phi\in K(z)$ if
$\phi^{-1}(\xi) = \{\xi\}$.

\begin{lemma}
\label{lem:invarII}
Let $K$ be a non-archimedean field, and let $\phi\in K(z)$ be a
rational function of degree $d\geq 2$.  Then
\begin{enumerate}
\item 
$\phi$ has good reduction if and only if
the Gauss point $\zeta(0,1)$ is totally invariant under $\phi$.
\item 
$\phi$ has potentially good reduction if and only if
there is some type~II point $\xi\in\PBerk$ that
is totally invariant under $\phi$.
In that case, $\xi$ is the only totally invariant
type~II point in $\PBerk$.
\end{enumerate}
\end{lemma}

\begin{proof}
This is Th\'{e}or\`{e}me~5 of \cite{Riv3}.
See also Corollary~9.27(C) and Proposition~10.45 of \cite{BR}.
\end{proof}

\begin{lemma}
\label{lem:btmap}
Let $a_1,a_2,b_1,b_2,c_1,c_2\in\Cv$, and let $\xi_1,\xi_2\in\PBerk$
be type~II points.  Suppose that $a_1$, $b_1$, and $c_1$ all lie in
different components of $\PBerk\smallsetminus\{\xi_1\}$,
and that $a_2$, $b_2$, and $c_2$ all lie in
different components of $\PBerk\smallsetminus\{\xi_2\}$.
Let $h\in\PGL(2,\Cv)$ be the unique linear fractional map
with $h(a_1)=a_2$, $h(b_1)=b_2$, and $h(c_1)=c_2$.
Then $h(\xi_1)=\xi_2$.
\end{lemma}

\begin{proof}
Suppose $h(\xi_1)\neq \xi_2$.  Let $U$ be the component of
$\PBerk\smallsetminus\{h(\xi_1)\}$ containing $\xi_2$,
and let $V$ be the component of $\PBerk\smallsetminus\{\xi_2\}$
containing $h(\xi_1)$.
By the continuity of the map $h^{-1}$,
the points $a_2$, $b_2$, $c_2$ must lie in separate
components of $\PBerk\smallsetminus\{h(\xi_1)\}$.
By hypothesis, they also lie in separate components of
$\PBerk\smallsetminus\{\xi_2\}$.  Thus, $U$ and $V$
can each contain at most one of these three points; without loss,
$c_2\in\PBerk\smallsetminus (U\cup V)$.

Since $c_2\not\in V$, the unique arc $\gamma\subseteq\PBerk$
from $c_2$ to $h(\xi_1)$ must pass through $\xi_2$.
But then the unique arc $\gamma'\subseteq\PBerk$ from $c_2$ to $\xi_2$
does \emph{not} pass through $h(\xi_1)$, and hence $c_2\in U$,
a contradiction.
\end{proof}

The next Lemma concerns the dynamics of $\phi$ on the components of
$\PBerk\smallsetminus\{\xi\}$ when $\xi$ is totally invariant.
It is merely a weak version of a
special case of Rivera-Letelier's far more general
Classification Theorem \cite{Riv1} in the context of
potentially good reduction.
However, its proof is much simpler than that of the
full Classification Theorem, and the statement below
will suffice for our purposes.

\begin{lemma}
\label{lem:juanclass}
Let $K$ be a non-archimedean field,
let $\phi\in K(z)$ be a nonconstant rational function,
and suppose that the type~II point $\xi\in\PBerk$ is totally invariant
under $\phi$.
Let $U$ be a component of $\PBerk\smallsetminus\{\xi\}$
such that $\phi(U)\subseteq U$.  Then
$\phi(U)=U$, and either
\begin{enumerate}
\item $\phi:U\to U$ is one-to-one, or
\item $\phi:U\to U$ is $\ell$-to-$1$, for
  some integer $\ell\geq 2$.
  Moreover, there is an attracting fixed point $a\in U\cap\PKbar$; and
  $\lim_{n\to\infty}\phi^n(x)=a$ for all $x\in U\cap\PCv$.
\end{enumerate}
\end{lemma}

\begin{proof}
After a $\Kbar$-rational change of coordinates, we may assume
that $\xi=\zeta(0,1)$, and that $U=D(0,1)$.
Expand $\phi(z)=\sum_{n\geq 0} c_n z^n\in \Kbar[[z]]$ as a power series
converging on $D(0,1)$.  Since $\phi(U)\subseteq U$ but also
$\phi(\zeta(0,1))=\zeta(0,1)$, we must have $\phi(D(0,1))=D(0,1)$.
In particular, $|c_n|\leq 1$ for all $n\geq 0$, with equality
for at least one $n$.
Let $\ell$ be the Weierstrass degree of $\phi$ on $D(0,1)$,
so that $1\leq \ell\leq d$.  If $\ell=1$, then we are in case~(a),
and we are done.

Otherwise, $\ell\geq 2$.  Thus,
$$|c_n|<1 \text{ for } n<\ell,
\quad |c_{\ell}|=1, \quad\text{and}\quad
|c_n|\leq 1 \text{ for }n>\ell.$$
Therefore, a glance at the Newton polygon of $\phi(z)-z$ shows
that $\phi$ has exactly one fixed point $a$ with $|a|<1$,
although \emph{a priori}, $a$ is 
defined over the completion $\Cv$
of $\Kbar$.  However, since $\phi\in\Kbar(z)$,
all the fixed points of $\phi$, including $a$, are defined over
$\Kbar$.

After a $\Kbar$-rational translation, we may assume that $a=0$,
so that $c_0=0$.  Thus, for any $x\in D(0,1)$,
$$|\phi(x)| \leq \max\{ |c_n x^n| : 1\leq n\leq \ell\}.$$
It follows that $\phi^n(x)\to 0$ as $n\to\infty$.
\end{proof}

Lemma~\ref{lem:juanclass} motivates the following definition.

\begin{defin}
\label{def:juanclass}
Let $K$ be a non-archimedean field,
let $\phi\in K(z)$ be a nonconstant rational function,
and suppose that the type~II point
$\xi\in\PBerk$ is totally invariant under $\phi$.
Let $U$ be a component of $\PBerk\smallsetminus\{\xi\}$.
If $\phi(U)\subseteq U$, we say $U$ is \emph{fixed} by $\phi$.

In that case, if $\phi:U\to U$ is one-to-one, then we say
$U$ is an \emph{indifferent} component; otherwise, we say
$U$ is an \emph{attracting} component.
\end{defin}

%In fact, if $U$ is a fixed component, then $\phi(U)=U$.
%Much more can be said about this classification (see \cite{Riv1}),
%but the properties above will suffice for our purposes.

If $K$ is complete and $U$ is an attracting
fixed component containing a $K$-rational point $x$,
then the attracting fixed point $a\in U$ of Lemma~\ref{lem:juanclass}.b
must also be $K$-rational, as it is a limit of iterates of $x$.
Thus, each of the remaining fixed points is defined over an extension
of $K$ of degree at most $d$, improving the \emph{a priori}
bound of $d+1$.
Although the same conclusion does not necessarily hold when
$K$ is not complete, it almost does, as the next lemma
makes precise.

\begin{lemma}
\label{lem:attrapprox}
Let $K$ be a non-archimedean field,
% of residue characteristic $p\geq 0$,
let $\phi\in K(z)$ be a rational function of degree $d\geq 2$,
and let $\xi\in\PBerk$ be a type~II point that is totally invariant
under $\phi$.  Suppose that $a\in\PK$ lies
in an attracting fixed component of $\PBerk\smallsetminus\{\xi\}$.
Then there is a polynomial $f\in K[z]$ such that
\begin{enumerate}
\item $\deg f = d$, and
\item none of the roots of $f$ in $\PCv$ lie in the same component
  of $\PBerk\smallsetminus\{\xi\}$ as $a$.
\end{enumerate}
\end{lemma}

\begin{proof}
After a $K$-rational change of coordinates if necessary, we may assume that
$\infty$ is \emph{not} one of the $d+1$ fixed points of $\phi$.
Let $U$ be the component of $\PBerk\smallsetminus\{\xi\}$
containing $a$.

Write $\phi=g/h$ with $g,h\in K[z]$ in lowest terms.  Then the
fixed points of $\phi$ in $\PCv$ are precisely the roots of
the polynomial
$$F(z) = h(z)\big(\phi(z) -z\big) = g(z) - zh(z)\in K[z].$$
Because $\infty$ is not fixed, we have $\deg F = d+1$.
By Lemma~\ref{lem:juanclass}.b, there is an attracting fixed point
$b\in U\cap\Kbar$ such that $\lim_{n\to\infty}\phi^n(a)=b$.
Denote the roots of $F$ by $b=b_0,b_1,\ldots,b_d\in\Kbar$.

Choose $\eps>0$ small enough that for each $j=0,\ldots,d$,
the disk $D(b_j,\eps)$ 
%contains no roots of $F$ distinct from $b_j$, and such that $D(b_j,\eps)$
is contained in a single component of $\PBerk\smallsetminus\{\xi\}$.
For each $n\geq 0$, define
$$G_n(z):=F(z) - F\big(\phi^n(a)\big)\in K[z],$$
which is a slight perturbation of $F$.
By Proposition~3.4.1.1 of~\cite{BGR},
for $n$ large enough, and hence for $\phi^n(a)$ close enough to $b$,
$G_n$ has exactly one root, namely $\phi^n(a)$,
in $D(b_0,\eps)$, and its other $d$ roots in
$\bigcup_{j=1}^d D(b_j,\eps)$.  Thus, the polynomial
$$f(z):= \frac{G_n(z)}{z - \phi^n(a)} \in K[z]$$
has degree $d$, with none of its roots lying in $U$.
\end{proof}

%\begin{lemma}
%\label{lem:goodunique}
%Let $K$ be a non-archimedean field, and let $\phi\in K(z)$ be a
%rational function of degree $d\geq 2$.  Let $\zeta(a,r)$ be a
%type~II point in $\PBerk$.  Suppose $h_1,h_2\in\PGL(2,K)$
%are linear fractional maps such that
%\begin{enumerate}
%\item $h_1(\zeta(a,r))=\zeta(0,1)$, and
%\item $h_1\circ \phi \circ h_1^{-1}$ and
%$h_2\circ \phi \circ h_2^{-1}$ have good reduction.
%\end{enumerate}
%Then $h_2(\zeta(a,r))=\zeta(0,1)$.
%\end{lemma}
%
%\begin{proof}
%START HERE!!!!
%\end{proof}

\begin{lemma}
\label{lem:pextn}
Let $K$ be a non-archimedean field of residue characteristic $p\geq 0$,
%and algebraic closure $\overline{K}$,
and let $f\in K[z]$ be a polynomial of degree $n\geq 1$.
Let $V\subseteq \Cv$ be a closed disk
that contains all the roots of $f$, and
let $e=v_p(n)\geq 0$.  Then
%Write $n=mp^e$, where $e\geq 0$, $m\geq 1$, and $p\nmid m$.  Then
there is a point $\alpha\in V$ such that $[K(\alpha):K]\leq p^e$.
%there is a field extension $L/K$ in $\overline{K}$
%such that $[L:K]\leq p^e$ and $V\cap L\neq\varnothing$.
\end{lemma}

When $p=0$, recall our convention that $v_p(n)=0$ and $p^0=1$.

\begin{proof}
Without loss, $f(z)=z^n + a_{n-1}z^{n-1} + \cdots + a_0$ is monic.
Let $q=p^e$, and
let $g(t)\in K[t]$ be the coefficient of $z^{n-q}$ in $f(z+t)\in K[t,z]$,
that is,
$$g(t) = \binom{n}{q} t^{q} + \sum_{j=0}^{q - 1}
a_{n-q + j} \binom{n-q + j}{n-q} t^j.$$
It is well known that $p\nmid \binom{n}{q}$
(this fact follows easily from the formula for $\ord_p(n!)$
found in, for example, Section V.3.1 of \cite{Rob}),
and hence that $|\binom{n}{q}|=1$ in $K$.

Let $\alpha\in\overline{K}$ be a root of $g$,
%and let $L:=K(\gamma)$,
so that $[K(\alpha):K]\leq \deg g = q$.
Let $h(z):=f(z+\alpha)\in K(\alpha)[z]$,
and let $W:=V-\alpha\subseteq\Cv$,
which is a closed disk containing all the roots of $h$.
It suffices to show that $0\in W$.

We proceed by contradiction.
If $0\not\in W$, then denoting the roots of $h$ by
$\beta_1,\ldots,\beta_n\in\overline{K}$,
we must have $|\beta_j - \beta_1|<|\beta_1|$ for all $j=1,\ldots, n$.
Therefore, for any $1\leq j_1 < \ldots < j_q \leq n$,
\begin{equation}
\label{eq:betabound}
\big| \beta_{j_1} \cdots \beta_{j_q} - \beta_1^q \big|
< |\beta_1|^q.
\end{equation}
%Set $q:=p^e$.
By construction, the $z^{n-q}$-coefficient of $h(z)$
is $0$. Thus, examining the same coefficient
when writing $h(z)=\prod_{j=1}^n (z-\beta_j)$, we have
\begin{equation}
\label{eq:zerocontra}
0 = (-1)^q \sum_{1\leq j_1 < \ldots < j_{q} \leq n}
\beta_{j_1}\cdots \beta_{j_{q}}
= (-1)^q \binom{n}{q} \beta_1^{q} + \eps,
\end{equation}
where $|\eps| < |\beta_1|^{q}$, by inequality~\eqref{eq:betabound}.
Since $|\binom{n}{q}|=1$, then, the right side of
equation~\eqref{eq:zerocontra} cannot be zero,
giving a contradiction.
\end{proof}

\begin{lemma}
\label{lem:diskscale}
Let $K$ be a non-archimedean field,
%with algebraic closure $\overline{K}$,
let $\phi\in K(z)$, and let $r>0$.
Suppose that $\phi$ maps $D(0,r)$ $\ell$-to-$1$ onto itself,
for some integer $\ell\geq 2$.
Then $r^{\ell-1}\in |K^{\times}|$.
\end{lemma}

\begin{proof}
Expand $\phi(z)=\sum_{n\geq 0} c_n z^n\in K[[z]]$ as a power series
converging on $D(0,r)$.  By hypothesis, the Weierstrass degree of
$\phi$ on $D(0,r)$ is $\ell$, and $\phi(D(0,r))=D(0,r)$.
In particular,
$$|c_\ell|r^\ell = r.$$
As a result, $c_{\ell}$ cannot be zero,
and $r^{\ell - 1} = |c_{\ell}|^{-1}\in |K^{\times}|$.
\end{proof}

\section{Proof of Theorem A}
\label{sec:proofA}

Theorem A rests on the following four results on the existence of
field extensions $L/K$ large enough to define new components of
$\PBerk\smallsetminus\{\xi\}$, where $\xi$ is the totally invariant
type~II point for a given map $\phi$ of potentially good reduction.
The basic idea underlying these results is twofold.  First,
an indifferent fixed component $U$ may contain multiple fixed points,
but for any $x$ in $U$, only one of the $d$ preimages of $x$ lies in $U$.
Second, an attracting fixed component $U$ may contain many, or even
all, of the preimages of a given point $x\in U$, but only one of the
$d+1$ fixed points.

\begin{thm}
\label{thm:indiffext}
Let $K$ be a non-archimedean field of residue characteristic $p\geq 0$,
and let $\phi\in K(z)$ be a rational function
of degree $d\geq 2$ with a totally invariant 
type~II point $\xi\in\PBerk$.
Let $e=v_p(d-1)\geq 0$, and
let $U$ be an indifferent fixed component of
$\PBerk\smallsetminus\{\xi\}$
such that $U\cap\PK\neq\varnothing$.
Then there is a point $\alpha\in\PKbar\smallsetminus U$
such that $[K(\alpha):K]\leq p^e$.
\end{thm}

As before, when $p=0$, recall our convention that $v_p(n)=0$ and $p^0=1$.

\begin{proof}
After a $K$-rational change of coordinates, we may assume that
$\infty\in U$.  Thus, $V:=\PCv\smallsetminus U$ is a closed disk
contained in $\Cv$.

Counting multiplicity,
there are exactly $d$ preimages of $\phi(\infty)$ in $\PCv$.
However, because $\infty$ is one of them, and because
$\phi:U\to U$ is one-to-one by Lemma~\ref{lem:juanclass},
there must be $d-1$ such preimages in $V$,
and they must be the roots of some polynomial
$f(z)\in K[z]$ of degree $d-1$.
By Lemma~\ref{lem:pextn}, then, there exists
a point $\alpha\in V$
such that $[K(\alpha):K]\leq p^e$.
\end{proof}

\begin{thm}
\label{thm:attrext}
Let $K$ be a non-archimedean field of residue characteristic $p\geq 0$,
and let $\phi\in K(z)$ be a rational function
of degree $d\geq 2$ with a totally invariant 
type~II point $\xi\in\PBerk$.
Let $e=v_p(d)\geq 0$, and
let $U$ be an attracting fixed component of
$\PBerk\smallsetminus\{\xi\}$
such that $U\cap\PK\neq\varnothing$.
Then there is a point $\alpha\in\PKbar\smallsetminus U$
such that $[K(\alpha):K]\leq p^e$.
%and $\phi(\alpha)\in U\cap\PK$.
\end{thm}

\begin{proof}
After a $K$-rational change of coordinates, we may assume that
$\infty\in U$.  Thus, $V:=\PCv\smallsetminus U$ is a closed disk
contained in $\Cv$.

According to Lemma~\ref{lem:attrapprox}, with $a=\infty$,
there is a polynomial $f\in K[z]$ of degree $d$ and with all its
roots in the closed disk $V$.
By Lemma~\ref{lem:pextn}, then, there exists
a point $\alpha\in V$ such that $[K(\alpha):K]\leq p^e$.
\end{proof}

\begin{thm}
\label{thm:twocomps1}
Let $K$ be a non-archimedean field,
and let $\phi\in K(z)$ be a rational function
of degree $d\geq 2$ with a totally invariant 
type~II point $\xi\in\PBerk$.
Let $U$ and $V$ be distinct components of 
$\PBerk\smallsetminus\{\xi\}$
that both contain points of $\PK$
and such that $U$ is fixed, while $\phi(V)\subseteq U$.
Then there is a point $\alpha\in\PKbar\smallsetminus (U\cup V)$
such that $[K(\alpha):K]\leq d$.
%and $\phi(\alpha)\in U\cap\PK$.
\end{thm}

\begin{proof}
After a $K$-rational change of coordinates, we may assume
that $0\in U$ and $\infty\in V$.  There are $d$ preimages of $\infty$
in $\PCv$, none of which can be in $U\cup V$, since
$\phi(U\cup V)\subseteq U$, whereas $\infty\not\in U$.
Thus, there is a polynomial $g(z)\in K[z]$ of degree $d$
with all of its roots in $\PCv\smallsetminus (U\cup V)$;
$g$ is, of course, simply the denominator of $\phi$.
Choosing $\alpha$ to be one of these roots, we have
$[K(\alpha):K]\leq\deg g = d$.
\end{proof}

\begin{thm}
\label{thm:twocomps2}
Let $K$ be a non-archimedean field,
and let $\phi\in K(z)$ be a rational function
of degree $d\geq 2$ with a totally invariant 
type~II point $\xi\in\PBerk$.
Let $U$ and $V$ be distinct components of 
$\PBerk\smallsetminus\{\xi\}$
that both contain points of $\PK$.  If either or both
of the following two conditions hold:
\begin{enumerate}
\item $U$ and $V$ are both fixed, or
\item $U$ is attracting fixed,
\end{enumerate}
then there is a point $\alpha\in\PKbar\smallsetminus (U\cup V)$
such that $[K(\alpha):K]\leq d-1$.
%and $\phi(\alpha)\in U\cap\PK$.
\end{thm}

\begin{proof}
Under either condition, $U$ is fixed.
After a $K$-rational change of coordinates, we may assume
that $0\in U$ and $\infty\in V$.
We consider two cases.

\textbf{Case 1}: $U$ is attracting.  Since $0\in U$ and
$\infty\in V$, it must be that
$U\cap\PCv=D(0,r)$ and
$V\cap\PCv = \PCv\smallsetminus\Dbar(0,r)$,
for some $r>0$.
By Lemma~\ref{lem:juanclass}, 
there is an integer $\ell\geq 2$ such that $\phi:U\to U$ is
$\ell$-to-$1$.  Because $\deg\phi=d$, we must have $\ell\leq d$.
By Lemma~\ref{lem:diskscale}, there is some $c\in K^{\times}$
such that $r^{\ell -1} = |c|$.  Let $\alpha\in\Kbar$ be
a root of the polynomial $z^{\ell-1} - c = 0$,
so that $|\alpha|=r$.  Hence,
$\alpha\in\PKbar\smallsetminus (U\cup V)$,
and $[K(\alpha):K]\leq \ell-1 \leq d-1$.

\textbf{Case 2}: $U$ is not attracting.  Then by hypothesis,
$U$ and $V$ are both fixed.  In addition, by Lemma~\ref{lem:juanclass},
$U$ is indifferent, and hence $\phi:U\to U$ is one-to-one.
The equation $\phi(z)=\phi(0)$ has exactly $d$ solutions in $\PCv$,
none of which can be in $V$ (since $V$ is fixed, and $\phi(0)\in U$).
Thus, this equation becomes a polynomial equation $g(z)=0$,
where $g\in K[z]$ has degree exactly $d$.  Clearly $0$ is a root of $g$,
and hence $f(z):=g(z)/z \in K[z]$ is a polynomial of degree $d-1$.

Let $\alpha\in\Kbar$ be a root of $f$, 
so that $[K(\alpha):K]\leq \deg f = d-1$.
Note that $\alpha$ cannot lie in $U$, since $\phi:U\to U$ is injective,
and the root $z=0$ of
$\phi(z)=\phi(0)$ is already in $U$.
It also cannot lie in $V$, as noted above.
Thus, $\alpha\in\PKbar\smallsetminus(U\cup V)$, as desired.
%
%\textbf{Case 3}: Otherwise.  By Lemma~\ref{lem:juanclass},
%both components $U$ and $V$ are attracting, and hence
%eachs contain a $K$-rational attracting fixed point.
%By Lemma~\ref{lem:attrapprox} with $\ell=2$, $x_1=0$, and $x_2=\infty$,
%there is a polynomial $f\in K[z]$ of degree $d-1\geq 1$ with all its
%roots in $\PKbar\smallsetminus (U\cup V)$.
%Choosing $\alpha$ to be one of these roots, we have
%$[K(\alpha):K]\leq\deg f\leq d-1$, and we are done.
\end{proof}

\begin{proof}[Proof of Theorem~A]
Assume that $\phi$ has potentially good reduction.  
By Lemma~\ref{lem:invarII}, there is a type~II point $\xi$
that is totally invariant under $\psi$.  We consider several cases.

\textbf{Case 1}:
$\PP^1(K)$ intersects at least three different components of
$\PBerk\smallsetminus\{\xi\}$.
Choose $a,b,c\in K$ in separate components, let
$h\in\PGL(2,K)$ be the unique linear fractional map
with $h(a)=0$, $h(b)=\infty$, and $h(c)=1$,
and let $\psi:=h\circ\phi\circ h^{-1}\in K(z)$.
By Lemma~\ref{lem:btmap}, $h(\xi)=\zeta(0,1)$.
Thus, $\zeta(0,1)$ must be totally invariant under $\psi$;
by Lemma~\ref{lem:invarII}, $\psi$ has good reduction.
Since $h$ is defined over $K$, and
$[K:K]=1 < B$, we are done.

\textbf{Case 2}:
$\PP^1(K)$ intersects exactly two components, $U$ and $V$, of
$\PBerk\smallsetminus\{\xi\}$, where $U$ is indifferent fixed
but $V$ is not fixed.
By continuity, $\phi(V)$ must be contained in some component,
and since $V\cap\PP^1(K)\neq\varnothing$, that component must
also contain $K$-rational points.  However,
by assumption, $\phi(V)\not\subseteq V$.
Thus, $\phi(V)\subseteq U$.

By Theorem~\ref{thm:twocomps1}, there exists an extension $L/K$
in $\Kbar$ such that $[L:K]\leq d$, and
$\PP^1(L)$ intersects at least three
different components of $\PBerk\smallsetminus\{\xi\}$.
By Case~1, then, $\phi$ is conjugate over $L$ to a map
of good reduction.  Since $d < B$, we are done.

\textbf{Case 3}:
$\PP^1(K)$ intersects exactly two components, $U$ and $V$, of
$\PBerk\smallsetminus\{\xi\}$, and \emph{either} $U$ and $V$
are both fixed, \emph{or} $U$ is attracting fixed.
By Theorem~\ref{thm:twocomps2}, there is an extension $L/K$
in $\Kbar$ such that $[L:K]\leq d-1$
and $\PP^1(L)$ intersects at least three
different components of $\PBerk\smallsetminus\{\xi\}$.
By Case~1, then, $\phi$ is conjugate over $L$ to a map
of good reduction.  Since $d-1 < B$, we are done.

\textbf{Case 4}:
$\PP^1(K)$ intersects exactly two components, $U$ and $V$, of
$\PBerk\smallsetminus\{\xi\}$, but neither is fixed.
In particular, neither can contain any fixed points in $\PKbar$.
The fixed points, however, are roots of a polynomial $f\in K[z]$
of degree $d+1$.  (After all,
we may make a $K$-rational change of coordinates
if necessary to guarantee that $\infty$ is not fixed.)
Let $\alpha$ be one of the fixed points, and let $L:=K(\alpha)$.
Then $[L:K]\leq \deg f =d+1$, and $\PP^1(L)$
intersects at least three different
components of $\PBerk\smallsetminus\{\xi\}$,
namely $U$, $V$, and the component containing $\alpha$.
By Case~1, then, $\phi$ is conjugate over $L$ to a map
of good reduction.  Since $d +1 \leq B$, we are done.

\textbf{Case 5}:
$\PP^1(K)$ intersects exactly one component $U$ of
$\PBerk\smallsetminus\{\xi\}$, and $U$ is an indifferent
fixed component.  Let $e:=v_p(d-1)$.
%Write $d=1+mp^e$, where $e\geq 0$, $m\geq 1$, and $p\nmid m$.
By Theorem~\ref{thm:indiffext}, there exists
$\alpha\in\PKbar\smallsetminus U$ such that $[K(\alpha):K]\leq p^e$.
Let $V$ be the component of 
$\PBerk\smallsetminus\{\xi\}$ containing $\alpha$.
Then with respect to the field $K(\alpha)$, we are now
in either Case~2 or Case~3 above, and hence there is
an extension $L/K(\alpha)$  with $[L:K(\alpha)]\leq d$
and such that $\phi$ is conjugate over $L$ to a map of good reduction.
Thus, $[L:K]\leq dp^e\leq B$, and we are done.

\textbf{Case 6}:
$\PP^1(K)$ intersects exactly one component $U$ of
$\PBerk\smallsetminus\{\xi\}$, but $U$ is not indifferent fixed.
Of course, $U$ must still be fixed, since
the image $\phi(U)$ must
be contained in a component but also contain $K$-rational points,
and hence $\phi(U)\subseteq U$.
By Lemma~\ref{lem:juanclass}, then, $U$ must be attracting fixed.

Let $e:=v_p(d)$.
%Write $d=mp^e$, where $e\geq 0$, $m\geq 1$, and $p\nmid m$.
By Theorem~\ref{thm:attrext}, there is a point
$\alpha\in\PKbar\smallsetminus U$ such that $[K(\alpha):K]\leq p^e$.
Let $V$ be the component of 
$\PBerk\smallsetminus\{\xi\}$ containing $\alpha$.
Then with respect to the field $K(\alpha)$, we are now
in Case~3 above, and hence there is
an extension $L/K(\alpha)$  with $[L:K(\alpha)]\leq d-1$
and such that $\phi$ is conjugate over $L$ to a map of good reduction.
Thus, $[L:K]\leq p^e(d-1)\leq B$, and we are done.
\end{proof}

\begin{remark}
In this Remark, we assume the reader is familiar with some
of the standard machinery of non-archimedean dynamics.
%see Chapter~10 of \cite{BR} or
%\cite{BenAZ} for some relevant background.

The Julia set $\calJ\subseteq\PBerk$
of a given rational function $\phi\in K(z)$
consists of precisely one point if and only if
$\phi$ has potentially good reduction;
see Th\'{e}or\`{e}me~4 of \cite{Riv3}
and Theorem~10.105 of \cite{BR}.
In that case, the one Julia point is the totally invariant type~II point of
Lemma~\ref{lem:invarII}.
The field $L$ in our Theorem~A therefore has the
property that the (one-point) Julia set of $\phi$ has a point
in the convex hull of $\PP^1(L)$. 
%The referee posed to us
Consider
the following related question:
for an arbitrary rational function $\phi\in K(z)$,
can we bound the degree of an extension $L/K$ such that the Julia set
of $\phi$ has a point in the convex hull of $\PP^1(L)$?

The answer is yes, with a bound of $d=\deg\phi$, if $K$ is complete,
or $d+1$ otherwise.
The fact that this
is smaller than our earlier bound $B$ does not contradict Theorem~B, as 
the convex hull condition does not force the diameter of
%the disk corresponding to
the type~II Julia point to lie in $|L|$. Here is the proof.

If $\PP^1(K)$ intersects the Julia set $\calJ$ of $\phi$, then we are done
by setting $L=K$.  Thus, we may assume that $\PP^1(K)$ is contained
in the Fatou set $\calF$ of $\phi$.  If $\PP^1(K)$ intersects more than one
component of $\calF$, then choosing $a,b\in\PP^1(K)$
in different components, the unique arc in $\PBerk$ from $a$ to $b$
must intersect $\calJ$, and we are again done by setting $L=K$.
We may therefore assume that $\PP^1(K)$ is contained in a single
connected component $U$ of $\calF$.  Since $\phi(0)\in\PP^1(K)$, this
component must be fixed.

According to the Classification Theorem of \cite{Riv1},
%(see also Theorem~10.76 of \cite{BR}),
the component $U$ is either attracting or indifferent.
If it is attracting and $K$ is complete, then the component $U$
contains a unique attracting fixed point $a=\lim_{n\to\infty}\phi^n(0)$,
which lies in $\PP^1(K)$ by our completeness assumption.
Meanwhile, $\phi$ must have $d$ other fixed points (counting multiplicity),
all of which must be distinct from $a$, since $a$ is attracting and
therefore does not have multiplier~$1$.  Let $\alpha$ be one of these
other fixed points, and let $L=K(\alpha)$,
so that $[L:K]\leq d$.  Since $\alpha$ cannot lie in $U$,
the unique arc in $\PBerk$ from $\alpha$ to $0$ must intersect $\calJ$,
and we are done.

Similarly, if $U$ is attracting and $K$ is not complete, then choosing
the same point fixed point $\alpha$ and $L=K(\alpha)$, we have
$[L:K]\leq d+1$.

Finally, if the fixed component $U$ containing $\PP^1(K)$ is indifferent,
then $\phi:U\to U$ is injective.
Let $a=\phi(0)$.  Then $\phi^{-1}(a)\smallsetminus\{0\}$ is nonempty
(since $0\in U$ cannot be a critical point),
and so we may choose $\alpha\neq 0$ such that $\phi(\alpha)=a$.
Let $L=K(\alpha)$, so that $[L:K]\leq d$.
Since $\alpha$ cannot lie in $U$ (as $\phi:U\to U$ is injective),
the unique arc in $\PBerk$ from $\alpha$ to $0$ must intersect $\calJ$,
and we are done.
\end{remark}

\section{Proof of Theorem B}
\label{sec:proofB}

We will need the following result on the field of definition
of a type~II point.

\begin{thm}
\label{thm:IIfod}
Let $K$ be a non-archimedean field,
and let $a,b\in\Kbar$ with $b\neq 0$.
Let $j,m\geq 1$ be the smallest positive integers such that
$|a|^j\in |K|$ and $|b|^m\in |K|$.
Let $L/K$ be an extension in $\Kbar$, let $h\in\PGL(2,L)$,
and suppose that $h\big(\zeta(a,|b|)\big) = \zeta(0,1)$.
\begin{enumerate}
\item
If $|a|\leq |b|$,
then the ramification index of $L/K$,
and hence also $[L:K]$, must be divisible by $m$.
\item
If $|a|>|b|$,
then the ramification index of $L/K$,
and hence also $[L:K]$, must be divisible by $\lcm(j,m)$.
\end{enumerate}
\end{thm}

\begin{proof}
Let $c_1=h^{-1}(0)$, $c_2=h^{-1}(\infty)$, and $c_3=h^{-1}(1)$,
all of which lie in $\PL$.
Since $h:\PBerk\to\PBerk$ is a homeomorphism
mapping $\zeta(a,|b|)$ to $\zeta(0,1)$, and
since $0$, $\infty$, and $1$ all lie in separate components
of $\PBerk\smallsetminus\{\zeta(0,1)\}$,
the three points $c_1$, $c_2$, and $c_3$
must lie in separate components of $\PBerk\smallsetminus\{\zeta(a,|b|)\}$.
In particular, at most one can lie outside $\Dbar(a,|b|)$.
Without loss, then, $c_1,c_2\in\Dbar(a,|b|)$,
and since both lie in separate components of
$\PBerk\smallsetminus\{\zeta(a,|b|)\}$, we must have
$|b|=|c_2-c_1|\in |L^{\times}|$.
Thus, the ramification index
$[|L^{\times}| : |K^{\times}|]$
of $L/K$ is divisible by $m$.

If $|a|>|b|$, then $|a|=|c_1|\in |L^{\times}|$ as well.
Therefore, $[|L^{\times}| : |K^{\times}|]$
is also divisible by $j$, and hence by $\lcm(j,m)$.
\end{proof}

\begin{proof}[Proof of Theorem~B]
Let $K$, $B$, and $\pi$ be as in the statement of the Theorem.
We may assume that $|p|\leq |\pi|<1$.  After all, if $K$ is discretely
valued, then we may choose $\pi$ to be a uniformizer for $K$.
And if $K$ is not discretely valued, then $|K^{\times}|^B$ is dense
in $(0,\infty)$, and hence there is some $a\in K^{\times}$ such
that $|p|\leq |a|^B<1$.  Thus, replacing $\pi$ by $a^{Bn}\pi$
for an appropriate $n\in\ZZ$, we have $|p|\leq |\pi|<1$, as desired.

Let $\gamma\in\Kbar$ be a $B$-th root of $\pi$, and let $L=K(\gamma)$.
By the defining property of $\pi$, the extension $L/K$ is totally
ramified, with $[L:K]=B$.
We now consider three cases.

\textbf{Case 1:} $d\geq 3$ and $p|d$, so that $B=p^e (d-1)$,
where $d=mp^e$ with $e,m\geq 1$ and $p\nmid m$.  Define $q:=p^e$, and
$$\phi(z) := z + \big(\pi^{-1} z^q - 1 \big)^m \in K[z].$$
Clearly $\deg\phi = mq = d$.  Let $\alpha:=\gamma^{d-1}$
and $\beta:=\gamma^q$, so that $\alpha^q = \beta^{d-1}=\pi$.
We will show that $\phi$ is conjugate over $L$ to a map of good
reduction, but not over any field $\tilde{L}/K$
with $[\tilde{L}:K]<B=q(d-1)$.

Define $h(z) = \beta^{-m}(z - \alpha)$.  
(This conjugating map is chosen to move a fixed point of $\phi$
to $0$ and then to scale to make the polynomial monic.)
Then $h(\zeta(\alpha,|\beta|^m))=\zeta(0,1)$, and
\begin{align}
\label{eq:zeromodp}
h\big(\phi(h^{-1}(z))\big) & =
\beta^{-m}\Big[\beta^{m} z + \alpha +
\Big(\alpha^{-q}\sum_{j=1}^q \binom{q}{j} \alpha^{q-j} \beta^{mj} z^j\Big)^m
 -\alpha\Big]
\notag
\\
&= z + \Big(\sum_{j=1}^q \binom{q}{j} \alpha^{-j} \beta^{mj-1} z^j\Big)^m.
\end{align}
Note that the lead coefficient of the polynomial~\eqref{eq:zeromodp}
is $(\alpha^{-q} \beta^{d-1})^m = 1^m=1$.
Note also that $|\alpha^{-j}\beta^{mj-1}| = |\pi|^{E_j}$, where
$$E_j = -\frac{j}{q} + \frac{mj-1}{d-1} = \frac{j-q}{q(d-1)}
>\frac{-1}{d-1} > -1.$$
Therefore, for
for $1\leq j\leq q-1$, the coefficient of $z^j$ inside the $m$-th power
in~\eqref{eq:zeromodp} has absolute value
$$\Big| \binom{q}{j} \alpha^{-j} \beta^{mj-1} \Big|
\leq |p| \cdot |\pi|^{E_j} < |p| \cdot |\pi|^{-1} \leq 1.$$
Thus, \eqref{eq:zeromodp} is a monic polynomial in $\ints_L[z]$,
and hence it has good reduction.

Finally, by the uniqueness statement of Lemma~\ref{lem:invarII}.b,
given an extension $\tilde{L}/K$ in $\Kbar$ and a map
$\tilde{h}(z)\in\PGL(2,\Kbar)$ for which
$\tilde{h}\circ\phi\circ\tilde{h}^{-1}$ has good reduction,
we must have
$\tilde{h}(\zeta(\alpha,|\beta|^m))=\zeta(0,1)$.  However,
$$|\beta|^m = |\pi|^{m/(d-1)} < |\pi|^{m/d} = |\pi|^{1/q} = |\alpha|.$$
By Theorem~\ref{thm:IIfod}.b, then,
$[\tilde{L}:K]\geq \lcm(q,d-1) = q(d-1) = B$.

%It remains to show that $\phi$ is not conjugate to a map of good
%reduction over any field of smaller degree over $K$.  To see this,
%note that the conjugating map $h$ above takes the type~II point
%$\xi:=\zeta(\alpha, |\beta|^m)$ to the Gauss point $\zeta(0,1)$.
%By the uniqueness statement of Lemma~\ref{lem:invarII}.b,
%given any $\tilde{h}(z)\in\PGL(2,\Kbar)$
%for which $\tilde{h}\circ\phi\circ\tilde{h}^{-1}$ has good reduction,
%then, $\tilde{h}$ must also map $\xi$ to $\zeta(0,1)$. However,
%$$|\beta|^m = |\pi|^{m/(d-1)} < |\pi|^{m/d} = |\pi|^{1/q} = |\alpha|.$$
%By Lemma~\ref{lem:berkram}, then,
%for any field $\tilde{L}/K$ over which $\tilde{h}$ is defined,
%the ramification index of $\tilde{L}/K$ must be divisible
%by both $q$ and $d-1$.  Since $(q,d-1)=1$,
%it follows that this ramification index is in fact divisible by $q(d-1)$,
%and hence that $[\tilde{L}:K]\geq q(d-1)$.

\textbf{Case 2:} $d\geq 3$ and $p|(d-1)$, so that $B=p^e d$,
where $d=1+mp^e$ with $e,m\geq 1$ and $p\nmid m$.  Define $q:=p^e$, and
$$\phi(z) :=  z + \frac{\pi^{d-1}}{(z^q - \pi^{q-1})^m} \in K(z).$$
Writing the above expression as a single fraction shows that
$\deg\phi = 1+qm = d$.
Let $\alpha:=\gamma^d$ and $\beta=\gamma^q$, so that
$\alpha^q=\beta^d=\pi$.
We will show that $\phi$ is conjugate over $L$ to a map of good
reduction, but not over any field $\tilde{L}/K$
with $[\tilde{L}:K]<B=qd$.

Define $h(z) = \beta^{-(d-1)}(z - \alpha^{q-1})$.  
(This conjugating map is chosen to move a pole of $\phi$
to $0$ and then to scale to change the numerator of the second
term of $\phi$ from $\pi^{d-1}$ to $1$.)
Then $h(\zeta(\alpha^{q-1},|\beta|^{d-1}))=\zeta(0,1)$, and
\begin{align}
\label{eq:onemodp}
h\big(\phi(h^{-1}(z))\big) & =
\beta^{-(d-1)}\bigg[\beta^{d-1} z + 
\frac{\beta^{d(d-1)}}{\dsps \Big( \sum_{j=1}^q \binom{q}{j}
\alpha^{(q-j)(q-1)} \beta^{j(d-1)} z^j\Big)^m}
\bigg]
\notag
\\
&= z + \frac{1}{\dsps \Big( \sum_{j=1}^q \binom{q}{j}
\alpha^{(q-j)(q-1)} \beta^{-(q-j)(d-1)} z^j\Big)^m}.
\end{align}
Consider the coefficient 
of $z^j$ inside the $m$-power in the denominator of~\eqref{eq:onemodp}.
This coefficient is $1$ for $j=q$, and for $1\leq j\leq q-1$
has absolute value
$$\Big|\binom{q}{j}\Big| \big| \alpha^{q-1}\beta^{-(d-1)}\big|^{q-j}
= \Big|\binom{q}{j}\Big| |\pi|^{E_j} \leq |p| |\pi|^{E_j},$$
where
$$E_j = (q-j)\Big( \frac{q-1}{q} - \frac{d-1}{d} \Big)
= -\frac{(q-j)(d-q)}{dq} > -\frac{(d-q)}{d} > -1.$$
Thus, still for $1\leq j\leq q-1$, the $z^j$-coefficient in question
has absolute value strictly less than $|p|\cdot|\pi|^{-1}\leq 1$.
Hence, the denominator of~\eqref{eq:onemodp} is a monic
polynomial $g(z)\in\ints_L[z]$ with reduction $\bar{g}(z) = z^{d-1}$.
Thus, the rational function
in~\eqref{eq:onemodp} is $(zg(z) + 1)/g(z)$, which has 
reduction $(z^d+1)/z^{d-1}$, exhibiting the desired
good reduction.

Once again,
given an extension $\tilde{L}/K$ in $\Kbar$ and a
linear fractional map
$\tilde{h}(z)\in\PGL(2,\Kbar)$ for which
$\tilde{h}\circ\phi\circ\tilde{h}^{-1}$ has good reduction,
we must have $\tilde{h}(\zeta(\alpha^{q-1},|\beta|^{d-1}))=\zeta(0,1)$,
by the uniqueness statement of Lemma~\ref{lem:invarII}.b.
However,
$$|\beta^{d-1}| = |\pi|^{(d-1)/d} < |\pi|^{(q-1)/q} = |\alpha^{q-1}|.$$
By Theorem~\ref{thm:IIfod}.b, then,
$[\tilde{L}:K]\geq \lcm(d,q) = dq =B$.

\textbf{Case 3:} 
In the remaining case, we have $B=d+1$.
%(In fact, the example we are about to give will work in Cases~1 and~2
%as well if $\pi\in K$ can be chosen to have the desired property for
%$d+1$ rather than $B$.)
Define
$$\phi(z) :=  \frac{\pi}{z^d},$$
let $\beta=\gamma\in\overline{K}$ be a $(d+1)$-st root of $\pi$,
and let $L:=K(\beta)$, so that $[L:K]=d+1$.
Let $h(z)=\beta^{-1} z$.
Then $h(\zeta(0,|\beta|))=\zeta(0,1)$, and
$$h\circ\phi\circ h^{-1}(z) 
= \frac{\pi}{\beta^{d+1} z^d} = \frac{1}{z^d},$$
which has good reduction.

As before, given an extension $\tilde{L}/K$ in $\Kbar$ and a map
$\tilde{h}(z)\in\PGL(2,\Kbar)$ for which
$\tilde{h}\circ\phi\circ\tilde{h}^{-1}$ has good reduction,
we must have $\tilde{h}(\zeta(0,|\beta|))=\zeta(0,1)$.
By Theorem~\ref{thm:IIfod}.a, then, $[\tilde{L}:K]\geq d+1 = B$.
\end{proof}

\section{Theorem~A and Ramification}
\label{sec:ramify}

The examples in our proof of Theorem~B attain the sharpness of the
bound $B$ by forcing the extension $L/K$ to have ramification degree
at least $B$.  This observation motivates the following conjecture.

\begin{conj}
\label{conj:thmAram}
Let $K$, $\phi$, and $B$ be as in Theorem~A.  
Assume that the residue field $k$ of $K$ is perfect.
If $\phi$ has
potentially good reduction, then there is a totally ramified
extension $L/K$ with $[L:K]\leq B$ such that $\phi$ is
conjugate over $L$ to a map of good reduction.
\end{conj}

Conjecture~\ref{conj:thmAram} would be false without
the hypothesis that $k$ is perfect.
For example, let $K=k((t))$, where $k=\FF_p(s)$,
and define $\phi(z)=t^{1-p}(z^p - s) + z$.
It is easy to check that  $h^{-1}\circ\phi\circ h(z)= z^p+z$, where
$h(z)=tz + \sqrt[p]{s}$, which moves
$\zeta(0,1)$ to $\zeta(\sqrt[p]{s},|t|)$.
Any other coordinate change $\tilde{h}$ also moving
$\zeta(0,1)$ to $\zeta(\sqrt[p]{s},|t|)$
would have to be defined over a field $L$ with residue
field containing $\sqrt[p]{s}$.  Since $[k(\sqrt[p]{s}):k]=p$,
$L/K$ cannot be totally
ramified, as the residue field extension degree must be at least $p$.

A proof of Conjecture~\ref{conj:thmAram} could perhaps proceed
along the same lines as our proof of Theorem~A, provided we
could prove totally ramified versions of the four theorems
in Section~\ref{sec:proofA}.

This is definitely possible for
Theorems~\ref{thm:twocomps1} and~\ref{thm:twocomps2}.
In the proof of each,
after the change of coordinates to move the two $K$-rational points
to $0$ and $\infty$, the point $\alpha$ could be replaced by
some $\beta$ with $|\beta|=|\alpha|$
so that $K(\beta)/K$ is totally ramified
and $[K(\beta):K] \leq [K(\alpha):K]$.
Such $\beta\in\Kbar$ certainly exists; if $n\geq 1$ is
the smallest positive integer such that $|\alpha|^n\in |K|$,
then pick $b\in K$ such that $|b|=|\alpha|^n$,
and let $\beta$ be an $n$-th root of $b$.
By the minimality of $n$, the extension $K(\beta)/K$ 
must be totally ramified.

%Of course, $\beta$ need not lie in the original extension field $K(\alpha)$.
%For example, if $K=\QQ_2$ and $\alpha$ is a root of
%$f(z) = z^6 + 2z^3 + 4$, then letting $L=K(\sqrt[3]{2},\zeta_9)$
%be the Galois closure
%of $K(\alpha)$, the Galois group $G=\Gal(L/K)$ is
%$G\cong (C_3\times C_3)\rtimes C_2$.  Checking the subgroups
%of $G$ shows that the only intermediate field strictly between
%$K$ and $K(\alpha)$ is $K(\zeta_3)$, which is unramified over $K$.
%However, choosing $\beta=\sqrt[3]{2}$, the extension
%$K(\beta)/K$ is totally ramified, and $|\beta|=|\alpha|$.

%In any event, the point is that
%Theorems~\ref{thm:twocomps1} and~\ref{thm:twocomps2}
%can be strengthened to be used appropriately in a proof
%of Conjecture~\ref{conj:thmAram}.

The obstacle, then, is to prove that we can choose the field extensions
from Theorems~\ref{thm:indiffext} and~\ref{thm:attrext}
to be totally ramified.  Thus,
Conjecture~\ref{conj:thmAram} reduces to proving the
following extension of Lemma~\ref{lem:pextn}.

\begin{conj}
\label{conj:pextnram}
Let $K$, $f$, $V$, $p$, $n$, and $e$ be as in Lemma~\ref{lem:pextn}.
Assume that the residue field $k$ of $K$ is perfect.
Then there is a point $\alpha$ in the disk $V$ such that
$K(\alpha)/K$ is totally ramified, and $[K(\alpha):K]\leq p^e$.
\end{conj}

{\bf Acknowledgements.}
The author thanks Keith Conrad, Xander Faber, and the referee
for their helpful discussions and suggestions.
The author also
gratefully acknowledges the support of NSF grant DMS-1201341.

\end{document}